\documentclass[11pt]{amsart}
\usepackage{epsfig,amsmath,amssymb, xspace, graphicx, url, amscd, euscript, mathrsfs,color}
\input xy
\xyoption{all}
\usepackage{tikz}
\usetikzlibrary{intersections}
\usetikzlibrary{arrows}
\usetikzlibrary{decorations.pathmorphing}

\newtheorem{theorem}[subsection]{Theorem}

\newtheorem{proposition-definition}[subsection]
{Proposition-Definition}
\newtheorem{corollary}[subsection]{Corollary}
\newtheorem{lemma}[subsection]{Lemma}

\theoremstyle{definition}

\newcommand{\scr}[1]{\mathbf{\EuScript{#1}}}

\newcommand\cB{\mathcal{B}}

\newcommand\cL{\mathcal{L}}
\newcommand\cM{\mathcal{M}}

\begin{document}

\title{The Picard groups of the stacks $\scr Y_0(2)$ and $\scr Y_0(3)$}
\author[A.~Niles]{Andrew Niles}
\address{Department of Mathematics \& Computer Science\\
College of the Holy Cross\\
1 College Street\\
Worcester, MA 01610\\
USA}
\email{aniles@holycross.edu}
\date{\today}

\begin{abstract}
We compute the Picard group of the stack of elliptic curves equipped with a cyclic subgroup of order two, and of the stack of elliptic curves equipped with a cyclic subgroup of order three, over any base scheme on which 6 is invertible. This generalizes a result of Fulton-Olsson, who computed the Picard group of the stack of elliptic curves (with no level structure) over a wide variety of base schemes.
\end{abstract}

\maketitle

\section{Introduction}

Modular curves and their stack-theoretic counterparts appear widely in modern number theory and arithmetic geometry, yet many basic questions about these stacks remain unaddressed or only recently addressed in the literature. For example, the complete moduli stacks $\scr X(N)$, $\scr X_0(N)$ and $\scr X_1(N)$ of generalized elliptic curves equipped with Drinfeld level structure were not shown to be algebraic stacks over $\mathrm{Spec}(\mathbb{Z})$ until this was proved in the paper \cite{C}, though partial (in fact, nearly complete) results in this direction were given in \cite{DR} and \cite{KM1}. 

Another property of these stacks that has not been fully considered is their Picard groups. For example, consider the stack $\cM_{1,1}$. Its Picard group, over an algebraically closed field of characteristic $\neq 2,3$, was computed by Mumford in \cite{M}. However, its Picard group over more general base schemes does not seem to appear in the literature until the relatively recent paper \cite{FO}. 

In this note we consider the analogous question for the stack $\scr Y_0(N)$ classifying elliptic curves equipped with a cyclic order-$N$ subgroup, for $N = 2$ and $N = 3$ (note that $\scr Y_0(2)$ equals the stack $\scr Y_1(2)$ classifying elliptic curves equipped with a point of exact order $2$). These stacks share some key characteristics with $\cM_{1,1}$: over $\mathbb{Z}[1/6]$, their generic points have stabilizer group $\mu_2$, some of their points have larger automorphism groups $\mu_4$ or $\mu_6$, and their coarse spaces are genus-$0$ modular curves.

As with all modular curves, the Picard group of the coarse space $Y_0(N)$ is well-understood, but that of the stack has not been previously addressed. Specifically, the main result of \cite{FO} is the construction of a canonical isomorphism 
\begin{displaymath}
\mathbb{Z}/(12) \times \mathrm{Pic}(M_{1,1,S}) \rightarrow \mathrm{Pic}(\cM_{1,1,S}),
\end{displaymath}
where $S$ is any scheme satisfying either (1) $2$ is invertible on $S$ or (2) $S$ is reduced, and where $M_{1,1,S} \cong \mathbb{A}^1_S$ is the coarse moduli space of $\cM_{1,1,S}$. We prove similar results for $\scr Y_0(2)_S$ and $\scr Y_0(3)_S$, though necessarily over a more restrictive class of base scheme:

\begin{theorem}
Let $S$ be a scheme on which $6$ is invertible, and let $\scr Y_0(2)_S$ (resp.~$\scr Y_0(3)_S$) be the Deligne-Mumford stack over $S$ whose objects over an $S$-scheme $T$ are pairs $(E,G)$, where $E/T$ is an elliptic curve and $G \subset E(T)$ is a cyclic subgroup of order $2$ (resp.~order $3$). Let $Y_0(2)_S$ (resp.~$Y_0(3)_S$) be the coarse moduli space of $\scr Y_0(2)_S$ (resp.~$\scr Y_0(3)_S$). Then there are canonical isomorphisms:
\begin{align*}
\mathbb{Z}/(4) \times \mathrm{Pic}(Y_0(2)_S) & \rightarrow \mathrm{Pic}(\scr Y_0(2)_S) \\
\mathbb{Z}/(6) \times \mathrm{Pic}(Y_0(3)_S) & \rightarrow \mathrm{Pic}(\scr Y_0(3)_S)
\end{align*}
\end{theorem}

\subsection*{Conventions} Over a scheme $S$, $\scr Y_0(N)_S$ is the category fibered in groupoids over $S$, whose objects over an $S$-scheme $T$ are pairs $(E,G)$, where $E/T$ is an elliptic curve and $G$ is a $[\Gamma_0(N)]$-structure on $E$. To work in arbitrary characteristics $G$ must be taken to be a Drinfeld structure (see \cite{KM1} or \cite{C}); however, our results are specific to $N = 2,3$ and only hold over schemes on which $6$ is invertible, so $G$ may simply be viewed as a cyclic, order-$N$ subgroup of $E(T)$. Over any scheme $S$ on which $N$ is invertible, $\scr Y_0(N)_S$ is a Deligne-Mumford stack (see \cite[\S4.2]{DR}).

\section{Picard group of $\scr Y_0(2)$}

Let $S$ be a scheme on which $6$ is invertible, and consider the stack $\scr Y_0(2)_S$. Let 
\begin{displaymath}
f: (\scr E, G) \rightarrow \scr Y_0(2)_S 
\end{displaymath}
be the universal elliptic curve with $[\Gamma_0(2)]$-structure, and let 
\begin{displaymath}
\lambda := f_\ast \Omega^1_{\scr E/\scr Y_0(2)_S} 
\end{displaymath}
be the Hodge bundle on $\scr Y_0(2)_S$. 

An object of $\scr Y_0(2)_S$ over an affine $S$-scheme $T = \mathrm{Spec}(R)$ consists of a pair $(E,G)$, where $E/T$ is an elliptic curve and $G \subseteq E(T)$ is a $[\Gamma_0(2)]$-structure on $E$, i.e.~a cyclic subgroup of $E(T)$ of order $2$, say $G = \langle P \rangle$ where $P \in E(T)$ has exact order $2$. Write 
\begin{displaymath}
t: T \rightarrow \scr Y_0(2)_S 
\end{displaymath}
for the morphism corresponding to $(E,G)$. 

Suppose the elliptic curve $E$ is given by an equation
\begin{displaymath}
y^2 = 4x^3 - g_2 x - g_3,
\end{displaymath}
where $g_2,g_3 \in R$. (Since $6$ is invertible, any elliptic curve may be given by such an equation locally on the base; see \cite[2.5]{D}.) Any change of coordinates may be given by $x \mapsto u^2 x$, $y \mapsto u^3 y$ for some unit $u \in R^\times$. We have the invariant differential $\omega = \frac{dx}{y}$ (a section of $t^\ast \lambda$), and under such a change of coordinates $\omega \mapsto u^{-1} \omega$. Finally, we have the discriminant $\Delta = g_2^3 - 27 g_3^2$, and under such a change of coordinates $\Delta \mapsto u^{12} \Delta$. 

Since $P \in E(T)$ has exact order $2$, in coordinates it is of the form $P = (x_0,0)$, where $x_0 \in R$ is a root of the equation $4x^3 - g_2 x - g_3 = 0$.

\begin{lemma}
Let $(E,G) \in \scr Y_0(2)_S$ as above, where $G = \langle P \rangle$ and $P$ is given in coordinates by $(x_0,0)$. Then the discriminant $\Delta = g_2^3 - 27 g_3^2$ factors as
\begin{displaymath}
\Delta = (g_2 - 3x_0^2)(g_2 - 12 x_0^2)^2.
\end{displaymath}
\end{lemma}

\begin{proof}
This is elementary, since $g_3 = 4x_0^2 - g_2 x_0$. We have:
\begin{align*}
\Delta & = g_2^3 - 27 g_3^2 \\
& = g_2^3 - 27 (4x_0^2 - g_2 x_0 )^2 \\
& = g_2^3 - 27 g_2^2 x_0^2 + 216 g_2 x_0^4 - 432 x_0^6 \\
& = (g_2 - 3x_0^2)(g_2 - 12 x_0^2)^2 .
\end{align*}
\end{proof}

In particular, $D := g_2 - 3x_0^2$ is a unit in $R$, since it is a factor of the unit $\Delta$. Note that under a coordinate transformation $x \mapsto u^2 x$, $y \mapsto u^3 y$ we also have $x_0 \mapsto u^2 x_0$ and $g_2 \mapsto u^4 g_2$, so $D \mapsto u^4 D$.

\begin{corollary}
$D \omega^{\otimes 4}$ is independent of the choice of coordinates.
\end{corollary}

Since $D \omega^{\otimes 4}$ is independent of the choice of coordinates, and since $D$ is a unit in the base, this means that for every morphism $t: T = \mathrm{Spec}(R) \rightarrow \scr Y_0(2)_S$ defining a pair $(E,G)$, where the elliptic curve $E$ is given by an equation $y^2 = 4x^3 - g_2 x - g_3$ and where $G = \langle (x_0,0) \rangle$, we have a canonical trivialization of $t^\ast \lambda^{\otimes 4}$ defined by $D \omega^{\otimes 4}$, independent of the choice of coordinates. Therefore we have shown: 
\begin{corollary}\label{lambda4triv}
$D \omega^{\otimes 4}$ defines a trivialization of the line bundle $\lambda^{\otimes 4}$ on $\scr Y_0(2)_S$.
\end{corollary}

We are now ready to prove the main result of this section. Let $S$ be a scheme on which $6$ is invertible, and consider the stack $\scr Y_0(2)_S$. Let $p: \scr Y_0(2)_S \rightarrow Y_0(2)_S$ be the coarse space morphism. 
\begin{theorem}
Under the above assumptions, the homomorphism 
\begin{align*}
\mathbb{Z}/(4) \times \mathrm{Pic}(Y_0(2)_S) & \rightarrow \mathrm{Pic}( \scr Y_0(2)_S ) \\
(n, \cL) & \mapsto \lambda^{\otimes n} \otimes p^\ast \cL 
\end{align*}
is an isomorphism.
\end{theorem}

Note that the reason the group homomorphism in the theorem is well-defined is the canonical trivialization of $\lambda^{\otimes 4}$ constructed in Corollary \ref{lambda4triv}.

\begin{proof}
Let $k$ be an algebraically closed field of characteristic $\neq 2,3$, and let $E/k$ be an elliptic curve equipped with a $[\Gamma_0(2)]$-structure $G$. Then $\underline{\mathrm{Aut}}(E,G) = \underline{\mathrm{Aut}}(E) \cong \mu_2$ unless $E$ has $j$-invariant $0$ or $1728$. 

If $j(E) = 0$, then $E$ is isomorphic to the elliptic curve $y^2 = x^3 + 1$ and $\underline{\mathrm{Aut}}(E) \cong \mu_6$. If $G = \langle (x_0,0) \rangle$ and $\eta \in \mu_6$ is a generator, then $E[2] = \{ \infty, (x_0,0), (\eta^2 x_0,0), (\eta^4 x_0,0) \}$, and under the action of $\mu_6$ we have $\eta \ast (x,y) = (\eta^2 x, \eta^3 y)$, so $\eta \ast (x_0,0) = (\eta^2 x_0,0)$. In particular, the $[\Gamma_0(2)]$-structure $G$ is only preserved by $\mu_2 \subset \mu_6$, so $\underline{\mathrm{Aut}}(E,G) = \mu_2$.

If $j(E) = 1728$, then $E$ is isomorphic to the elliptic curve $y^2 = x^3 - x$ and $\underline{\mathrm{Aut}}(E) \cong \mu_4$. We have $E[2] = \{ \infty, (0,0), (-1,0), (1,0) \}$. If $i \in \mu_4$ is a generator, then $i \ast (x,y) = (-x,iy)$, so $i \ast (0,0) = (0,0)$ but $i \ast (-1,0) = (1,0)$. So the $[\Gamma_0(2)]$-structure $G = \langle (0,0) \rangle$ is fixed by all of $\mu_4$, hence $\underline{\mathrm{Aut}}(E,G) \cong \mu_4$; but the other two $[\Gamma_0(2)]$-structures $G_1 = \langle (-1,0) \rangle$ and $G_2 = \langle (1,0) \rangle$ are only preserved by $\mu_2 \subset \mu_4$.

We conclude that the coarse space morphism $p:\scr Y_0(2)_S \rightarrow Y_0(2)_S$ is a $\mu_2$-gerbe over 
\begin{displaymath}
Y_0(2)_S \setminus p(E,G),
\end{displaymath}
where $E$ is the elliptic curve $y^2 = x^3 - x$ and $G = \langle (0,0) \rangle$.

Let $\tilde{s}: S \rightarrow \scr Y_0(2)_S$ be the morphism corresponding to $(E,G)$, where $E$ is the elliptic curve $y^2 = x^3 - x$ and $G$ is the $[\Gamma_0(2)]$-structure $\langle (0,0) \rangle$. Since $\underline{\mathrm{Aut}}(E,G) \cong \mu_4$, this defines a closed immersion $s: \cB \mu_{4,S} \hookrightarrow \scr Y_0(2)_S$. For any line bundle $\cL$ on $\scr Y_0(2)_S$, $s^\ast \cL$ is a line bundle on $\cB \mu_{4,S}$, corresponding to a line bundle $L$ on $S$ equipped with an action of $\mu_4$. This means we have a map $\rho: \mu_4 \rightarrow \underline{\mathrm{Aut}}(L) \cong \mathbb{G}_m$, which corresponds to a character $\chi \in \mathbb{Z}/(4)$.

This defines a homomorphism 
\begin{align*}
\mathrm{Pic}(\scr Y_0(2)_S) & \rightarrow \mathbb{Z}/(4) \\
\cL & \mapsto \chi .
\end{align*}
Let $K$ denote the kernel. 

The pullback $p^\ast: \mathrm{Pic}(Y_0(2)_S) \rightarrow \mathrm{Pic}(\scr Y_0(2)_S)$ is injective and lands in $K$. We claim that 
\begin{displaymath}
p^\ast: \mathrm{Pic}(Y_0(2)_S) \rightarrow K 
\end{displaymath}
is in fact an isomorphism. By \cite[2.3]{FO}, it therefore suffices to show that for every geometric point $\overline{x} \rightarrow \scr Y_0(2)_S$ and every $\cL \in K$, the action of the stabilizer group of $\overline{x}$ on $\cL(\overline{x})$ is trivial; for then $p_\ast \cL$ is an invertible sheaf on $Y_0(2)_S$ and $p^\ast p_\ast \cL \rightarrow \cL$ is an isomorphism.

Indeed, if $\cL \in K$ then the action $\rho$ defined above is trivial. By \cite[$\textrm{proof of $2.2$}$]{FO}, $\rho|_{\mu_2}$ equals the action of $\mu_2$ on the fiber of $\cL$ at the generic point of $\scr Y_0(2)$, hence the action of $\mu_2$ on the fiber of $\cL$ at the generic point is trivial. Therefore the action is trivial at every point since $\scr Y_0(2)_S$ is a $\mu_2$-gerbe over $Y_0(2)_S \setminus p(E,G)$, for $E$ the elliptic curve $y^2 = x^3 - x$ and $G = \langle (0,0) \rangle$ as discussed above.

Therefore $p^\ast: \mathrm{Pic}(Y_0(2)_S) \rightarrow K$ is an isomorphism. Furthermore, the map 
\begin{displaymath}
\mathrm{Pic}(\scr Y_0(2)_S) \rightarrow \mathbb{Z}/(4) 
\end{displaymath}
is surjective, by the same argument as \cite[2.5]{FO}. Indeed, the image of $\lambda$ in $\mathbb{Z}/(4)$ is a generator, because the action of any $\zeta \in \mu_4$ on the invariant differential $\omega$ of the elliptic curve $y^2 = x^3 - x$ is given by multiplication by $\zeta$.

Therefore we have a short exact sequence 
\begin{displaymath}
0 \rightarrow \mathrm{Pic}(Y_0(2)_S) \stackrel{p^\ast}{\rightarrow} \mathrm{Pic}(\scr Y_0(2)_S) \rightarrow \mathbb{Z}/(4) \rightarrow 0.
\end{displaymath}
The homomorphism 
\begin{align*}
\mathbb{Z}/(4) & \rightarrow \mathrm{Pic}(\scr Y_0(2)_S) \\
n & \mapsto \lambda^{\otimes n} 
\end{align*}
(well-defined due to the canonical trivialization of $\lambda^{\otimes 4}$ provided by Corollary \ref{lambda4triv}) provides a canonical splitting.
\end{proof}

\section{Picard group of $\scr Y_0(3)$}

Let $S$ be a scheme on which $6$ is invertible, and consider the stack $\scr Y_0(3)_S$. Let 
\begin{displaymath}
f: (\scr E, G) \rightarrow \scr Y_0(3)_S 
\end{displaymath}
be the universal elliptic curve with $[\Gamma_0(3)]$-structure, and let 
\begin{displaymath}
\lambda := f_\ast \Omega^1_{\scr E/\scr Y_0(3)_S} 
\end{displaymath}
be the Hodge bundle on $\scr Y_0(3)_S$. 

An object of $\scr Y_0(3)_S$ over an affine $S$-scheme $T = \mathrm{Spec}(R)$ consists of a pair $(E,G)$, where $E/T$ is an elliptic curve and $G \subseteq E(T)$ is a $[\Gamma_0(3)]$-structure on $E$, i.e.~a cyclic subgroup of $E(T)$ of order $3$. Write 
\begin{displaymath}
t: T \rightarrow \scr Y_0(3)_S 
\end{displaymath}
for the morphism corresponding to $(E,G)$. 

Suppose the elliptic curve $E$ is given by an equation
\begin{displaymath}
y^2 = 4x^3 - g_2 x - g_3,
\end{displaymath}
where $g_2,g_3 \in R$. (As before, since $6$ is invertible, any elliptic curve may be given by such an equation locally on the base.) As before, any change of coordinates may be given by $x \mapsto u^2 x$, $y \mapsto u^3 y$ for some unit $u \in R^\times$. We again have the invariant differential $\omega = \frac{dx}{y}$ (a section of $t^\ast \lambda$), and under such a change of coordinates $\omega \mapsto u^{-1} \omega$. 

The group law on $E$ tells us that $G = \{ \infty, (x_0, \pm y_0) \}$ for some $x_0,y_0 \in R$. Furthermore, $y_0$ must be a unit in $R$; otherwise, modulo some prime ideal of $R$ we would have $G = \{ \infty, (x_0,0) \}$, which is impossible since $G$ must remain a cyclic order-$3$ subgroup modulo any such reduction.

In particular, the unit $(y_0)^2 = (-y_0)^2 \in R$ is canonically determined (in our coordinates) by the choice of $[\Gamma_0(3)]$-structure $G$ on $E$. Under a change of coordinates given by $x \mapsto u^2 x$ and $y \mapsto u^3 y$, we have $(y_0)^2 \mapsto u^6 (y_0)^2$. In particular, $(y_0)^2 \omega^{\otimes 6}$ is independent of the choice of coordinates, so it provides a canonical trivialization of $t^\ast \lambda^{\otimes 6}$. We have proven:
\begin{corollary}\label{lambda6triv}
$(y_0)^2 \omega^{\otimes 6}$ defines a canonical trivialization of the line bundle $\lambda^{\otimes 6}$ on $\scr Y_0(3)_S$.
\end{corollary}

We are now ready to prove the main result of this section. Let $S$ be a scheme on which $6$ is invertible, and consider the stack $\scr Y_0(3)_S$. Let $p: \scr Y_0(3)_S \rightarrow Y_0(3)_S$ be the coarse space morphism. 
\begin{theorem}
Under the above assumptions, the homomorphism 
\begin{align*}
\mathbb{Z}/(6) \times \mathrm{Pic}(Y_0(3)_S) & \rightarrow \mathrm{Pic}( \scr Y_0(3)_S ) \\
(n, \cL) & \mapsto \lambda^{\otimes n} \otimes p^\ast \cL 
\end{align*}
is an isomorphism.
\end{theorem}

\begin{proof}
Let $k$ be an algebraically closed field of characteristic $\neq 2,3$, and let $E/k$ be an elliptic curve equipped with a $[\Gamma_0(3)]$-structure $G$. Then $\underline{\mathrm{Aut}}(E,G) = \underline{\mathrm{Aut}}(E) \cong \mu_2$ unless $E$ has $j$-invariant $0$ or $1728$. 

If $j(E) = 0$, then $E$ is isomorphic to the elliptic curve $y^2 = x^3 + 1$ and $\underline{\mathrm{Aut}}(E) \cong \mu_6$. One may compute:
\begin{displaymath}
E[3] = \{ \infty, (0,\pm 1), (\sqrt[3]{-4}, \pm \sqrt{-3}), (\zeta \sqrt[3]{-4}, \pm \sqrt{-3}), (\zeta^2 \sqrt[3]{-4}, \pm \sqrt{-3} ) \},
\end{displaymath}
where $\zeta \in \mu_3^\times$. $\eta \in \mu_6$ acts on $E$ by $\eta \cdot (x,y) = (\eta^2 x,\eta^3 y)$, so we see that the $\mu_6$ action preserves the $[\Gamma_0(3)]$-structure $G = \{ \infty, (0,\pm 1) \}$, but every other $[\Gamma_0(3)]$-structure on $E$ is only preserved by $\mu_2 \subset \mu_6$.

If $j(E) = 1728$, then $E$ is isomorphic to the elliptic curve $y^2 = x^3 - x$ and $\underline{\mathrm{Aut}}(E) \cong \mu_4$. One may compute that the $x$-coordinate of any point of exact order $3$ on $E$ is of the form 
\begin{displaymath}
x_0 = \pm \sqrt{1 \pm \frac{\sqrt{3}}{3} }.
\end{displaymath}
For any fixed $x_0$ of this form, the action of a generator $i \in \mu_4$ is given by $i \cdot (x_0,y_0) = (-x_0,i y_0)$. In particular, none of the four $[\Gamma_0(3)]$-structures on $E$ are preserved by all of $\mu_4$. 

We conclude that the coarse space morphism $p:\scr Y_0(3)_S \rightarrow Y_0(3)_S$ is a $\mu_2$-gerbe over 
\begin{displaymath}
Y_0(3)_S \setminus p(E,G),
\end{displaymath}
where $E$ is the elliptic curve $y^2 = x^3 +1$ and $G = \{ \infty, (0,\pm 1) \}$. For this pair $(E,G)$ we have $\underline{\mathrm{Aut}}(E,G) \cong \mu_6$.

Let $\tilde{s}: S \rightarrow \scr Y_0(3)_S$ be the morphism corresponding to $(E,G)$, where $E$ is the elliptic curve $y^2 = x^3 + 1$ and $G$ is the $[\Gamma_0(3)]$-structure $\{ \infty, (0,\pm 1) \}$. Since $\underline{\mathrm{Aut}}(E,G) \cong \mu_6$, this defines a closed immersion $s: \cB \mu_{6,S} \hookrightarrow \scr Y_0(3)_S$. For any line bundle $\cL$ on $\scr Y_0(3)_S$, $s^\ast \cL$ is a line bundle on $\cB \mu_{6,S}$, corresponding to a line bundle $L$ on $S$ equipped with an action of $\mu_6$. This means we have a map $\rho: \mu_6 \rightarrow \underline{\mathrm{Aut}}(L) \cong \mathbb{G}_m$, which corresponds to a character $\chi \in \mathbb{Z}/(6)$.

This defines a homomorphism 
\begin{align*}
\mathrm{Pic}(\scr Y_0(3)_S) & \rightarrow \mathbb{Z}/(6) \\
\cL & \mapsto \chi .
\end{align*}
Let $K$ denote the kernel. 

The pullback $p^\ast: \mathrm{Pic}(Y_0(3)_S) \rightarrow \mathrm{Pic}(\scr Y_0(3)_S)$ is injective and lands in $K$. We claim that 
\begin{displaymath}
p^\ast: \mathrm{Pic}(Y_0(3)_S) \rightarrow K 
\end{displaymath}
is in fact an isomorphism. By \cite[2.3]{FO}, it therefore suffices to show that for every geometric point $\overline{x} \rightarrow \scr Y_0(3)_S$ and every $\cL \in K$, the action of the stabilizer group of $\overline{x}$ on $\cL(\overline{x})$ is trivial; for then $p_\ast \cL$ is an invertible sheaf on $Y_0(3)_S$ and $p^\ast p_\ast \cL \rightarrow \cL$ is an isomorphism.

Indeed, if $\cL \in K$ then the action $\rho$ defined above is trivial. By \cite[$\textrm{proof of $2.2$}$]{FO}, $\rho|_{\mu_2}$ equals the action of $\mu_2$ on the fiber of $\cL$ at the generic point of $\scr Y_0(3)$, hence the action of $\mu_2$ on the fiber of $\cL$ at the generic point is trivial. Therefore the action is trivial at every point since $\scr Y_0(3)_S$ is a $\mu_2$-gerbe over $Y_0(3)_S \setminus p(E,G)$, for $E$ the elliptic curve $y^2 = x^3 + 1$ and $G = \{ \infty, (0,\pm 1) \}$ as discussed above.

Therefore $p^\ast: \mathrm{Pic}(Y_0(3)_S) \rightarrow K$ is an isomorphism. Furthermore, the map 
\begin{displaymath}
\mathrm{Pic}(\scr Y_0(3)_S) \rightarrow \mathbb{Z}/(6) 
\end{displaymath}
is surjective. Indeed, the action of $\eta \in \mu_6$ on the elliptic curve $y^2 = x^3 + 1$ is given by $\eta \cdot (x,y) = (\eta^2 x,\eta^3 y)$, so its action on the invariant differential $\omega = \frac{dx}{y}$ is equal to multiplication by $\eta^{-1}$. So the image of $\lambda$ in $\mathbb{Z}/(6)$ is a generator.

Therefore we have a short exact sequence 
\begin{displaymath}
0 \rightarrow \mathrm{Pic}(Y_0(3)_S) \stackrel{p^\ast}{\rightarrow} \mathrm{Pic}(\scr Y_0(3)_S) \rightarrow \mathbb{Z}/(6) \rightarrow 0.
\end{displaymath}
The homomorphism 
\begin{align*}
\mathbb{Z}/(6) & \rightarrow \mathrm{Pic}(\scr Y_0(3)_S) \\
n & \mapsto \lambda^{\otimes n} 
\end{align*}
(well-defined due to the canonical trivialization of $\lambda^{\otimes 6}$ provided by Corollary \ref{lambda6triv}) provides a canonical splitting.
\end{proof}

\end{document}